\documentclass[twoside,a4paper,reqno,11pt]{amsart}
\usepackage{amsfonts, amsbsy, amsmath, amssymb, latexsym}
\usepackage{mathrsfs,array}
\usepackage[top=30mm,right=30mm,bottom=30mm,left=30mm]{geometry}
\usepackage{stmaryrd}
\usepackage{bm}
\usepackage{xcolor}
\usepackage{hyperref}

\newcommand{\N}{\mathbb{N}}
\newcommand{\R}{\mathbb{R}}

\newcommand{\Q}{\mathbb{Q}}
 \renewcommand{\to}{\rightarrow}

\DeclareMathOperator{\cd}{cd}

\DeclareMathOperator{\Gal}{Gal}

\makeatletter
\newcommand{\imod}[1]{\allowbreak\mkern4mu({\operator@font mod}\,\,#1)}
\makeatother

\newtheorem{theorem}{Theorem}
\newtheorem*{conj*}{Conjecture}
\newtheorem{lemma}[theorem]{Lemma}

\newtheorem{corollary}[theorem]{Corollary}

\newtheorem{thm}{Theorem}[section]

\newtheorem{cor}[thm]{Corollary}

\theoremstyle{definition}

\begin{document}

 \author{Damian Sercombe}
\address{D. Sercombe, Institute of Mathematics, Hebrew University, Jerusalem 91904, Israel}
\email{damian.sercombe@mail.huji.ac.il}

\author{Aner Shalev}
\address{A. Shalev, Institute of Mathematics, Hebrew University, Jerusalem 91904, Israel}
\email{shalev@math.huji.ac.il}

\title{The length and depth of associative algebras}

\begin{abstract}
Recently there has been considerable interest in studying the length and the depth of finite groups, algebraic groups and Lie groups.
In this paper we introduce and study similar notions for algebras.
Let $k$ be a field and let $A$ be an associative, not necessarily unital, algebra over $k$. An unrefinable chain of $A$ is a chain of subalgebras $A=A_0>A_1>...>A_t=0$ for some integer $t$ where each $A_i$ is a maximal subalgebra of $A_{i-1}$. The maximal (respectively, minimal) length of such an unrefinable chain is called the length (respectively, depth) of $A$. It turns out that finite length,
finite depth and finite dimension are equivalent properties for $A$. For $A$ finite dimensional, we give a formula for the length of $A$, we bound the depth of $A$, and we study when the length of $A$ equals its dimension and its depth respectively. Finally, we investigate under what circumstances the dimension of $A$ is bounded above by a function of its length, or its depth, or its length minus its depth.
 \end{abstract}

%\date{\today}


\thanks{DS was supported by a Post-Doctoral Fellowhip from ISF grant 686/17 of AS.
AS was partially supported by ISF grant 686/17 and the Vinik Chair of mathematics which he holds.}

% \thanks{\textcolor{red}{The second author acknowledges the support of ???}}

\maketitle

\section{Introduction}

\noindent Let $k$ be a field and let $A$ be an associative, not necessarily unital, algebra over $k$. An \textit{unrefinable} chain
of length $t$ of $A$ is a chain of subalgebras $A=A_0>A_1>...>A_t=0$ where each $A_i$ is a maximal $k$-subalgebra of $A_{i-1}$.

\vspace{2mm}\noindent The \textit{length} $l(A)$ of $A$ is the maximal length of an unrefinable chain. If there is no such a chain, or no unrefinable chain of maximal length, we set $l(A) = \infty$.

\vspace{2mm}\noindent The \textit{depth} $\lambda(A)$ of $A$ is the minimal length of an unrefinable chain. If there is no unrefinable chain we set
$\lambda(A) = \infty$. We clearly have $\lambda(A) \le l(A) \le \dim A$.

\vspace{2mm}\noindent These invariants were first introduced for finite groups in the 1960s (see \cite{BLS3} and the references therein for a comprehensive summary). More recently, length and depth have been introduced and studied for algebraic groups over algebraically closed fields in \cite{BLS2} and over $\R$ in \cite{BLS4,Se}; see also \cite{BLS1} which focuses on the depth of finite simple groups.

\vspace{2mm}\noindent Our first result is as follows.

\begin{thm}\label{finlen} The following are equivalent for an associative algebra $A$.

$(i)$ $\lambda (A) < \infty$.

$(ii)$ $l(A) < \infty$.

$(iii)$ $\dim A < \infty$.

\end{thm}

\noindent Henceforth we assume throughout this paper that all our algebras are finite dimensional. We continue to assume that our algebras are associative. Unless otherwise stated, all algebras, algebra homomorphisms and vector spaces are assumed to be over $k$.

\vspace{2mm}\noindent Define the \emph{chain difference} $\cd(A)$ of an algebra $A$ by
\[
\cd(A) = l(A) - \lambda(A).
\]

\vspace{2mm}\noindent In Theorem \ref{length} we obtain a formula for the length of any algebra. In Theorems \ref{depth} and \ref{depthfinitefields} we bound above and below the depth of a simple algebra. In Theorem \ref{dimthings} we study when the length of an algebra equals its dimension and its depth respectively. The question of bounding the dimension of an algebra in terms of its length, its depth or its chain difference is discussed in Theorems \ref{solvrad} and \ref{ld}.

\vspace{2mm}\noindent Let $D$ be a division algebra and $n$ a positive integer. Let $M_n(D)$ denote the algebra of $n \times n$ matrices with coefficients in $D$. % Observe that $Z\big(M_n(D)\big) =\{cI \hspace{-0.5mm}\mid \hspace{-0.5mm} c \in Z(D)\} \cong Z(D)$ where $I$ denotes the $r \times r$ identity matrix.
Let $T_n(D)$ (resp. $U_n(D)$) denote the subalgebra of upper (resp. strictly upper) triangular matrices of $M_n(D)$. If $k$ is algebraically closed then the only division algebra over $k$ is $k$ itself. Let $E$ be a division subalgebra of $D$. We say that $D$ is \textit{Galois} over $E$ if $E$ is the set of fixed elements of a group of automorphisms acting on $D$. Let $\Gal(D/E)$ denote the group of automorphisms of $D$ that fix $E$ pointwise. See \cite{W} for a Galois theory of division algebras.

\vspace{2mm}\noindent Let $A$ be an algebra. The Jacobson radical $J(A)$ of $A$ is a nilpotent ideal of $A$. If $J(A)$ is trivial then $A$ is \textit{semisimple}. % https://math.stackexchange.com/questions/2871063/nilpotent-jacobson-radical,     https://en.wikipedia.org/wiki/Semisimple_algebra,     https://www.encyclopediaofmath.org/index.php/finite dimensional_associative_algebra
The quotient $A/J(A)$ is a semisimple algebra. % https://en.wikipedia.org/wiki/Jacobson_radical
We denote the dimension of $A$ as a $k$-algebra by either $\dim A$ or $[A:k]$.
%There exists a faithful representation of $A$ in $M_d(k)$ for some integer $d$. % https://math.stackexchange.com/questions/2281006/representation-theorem-for-finite dimensional-algebra-over-fields

\vspace{2mm}\noindent By the well-known Wedderburn's Theorem (see for instance Theorems III.$8$ and III.$9$ of \cite{A} or Theorem $2.1.8$ of \cite{R}), there is an algebra isomorphism $A/J(A) \cong \prod_{i=1}^m M_{n_i}(D_i)$ for some integers $m$, $n_1$, ..., $n_m$ and division algebras $D_1$,..., $D_m$ that is unique up to permutation of the factors.
%In fact, we can choose a basis of $A$ such that $A/J(A) = \prod_{i=1}^m M_{n_i}(D_i)$.
If $n_i=1$ for each $i$ then $A$ is \textit{basic}. % p58 of https://www.imath.kiev.ua/~drozd/drozd-kirichenko_en.pdf
If $D_i \cong k$ for each $i$ then $A$ is \textit{split}. % p62 of https://www.imath.kiev.ua/~drozd/drozd-kirichenko_en.pdf
So if $A/J(A)$ is isomorphic to a direct product of copies of $k$ then $A$ is basic split. We consider (non-unital) nilpotent algebras to be basic split.

\vspace{2mm}\noindent  If $A/J(A)$ is a separable algebra (this always holds if $k$ is a perfect field) then there exists a semisimple subalgebra $S$ of $A$ such that $A=S \oplus J(A)$ as vector spaces. % If $S'$ is another subalgebra of $A$ satisfying $A=S' \oplus J(A)$ then $S'$ is conjugate to $S$ by an element of $1+J(A)$.
This result is called Wedderburn's Principal Theorem (Theorem III.$23$ of \cite{A} or Theorem $2.5.37$ of \cite{R}). % $A$ is not required to be unital, see https://arxiv.org/pdf/1609.05812.pdf
% The Block Theory of Finite Group Algebras p339
% Nice little note on when this theorem fails: https://math.stackexchange.com/questions/1930682/finite dimensional-algebras-which-do-not-satisfy-wedderburns-principal-theorem

\vspace{2mm}\noindent Now assume that $A$ is semisimple, where $A = \prod_{i=1}^m M_{n_i}(D_i)$. The \textit{rank} of $A$ is $r:=\sum_{i=1}^m (n_i-1)$. A \textit{Borel subalgebra} of $A$ is any conjugate of $\prod_{i=1}^m T_{n_i}(D_i)$. A \textit{parabolic subalgebra} of $A$ is any subalgebra that contains a Borel subalgebra. These are the analogues of Borel subgroups and parabolic subgroups of an algebraic group. Let $B$ be a Borel subalgebra of $A$. The \textit{parabolic length} of $A$ is defined by $l(B)+r$.

\begin{thm}\label{length} Let $A$ be an algebra, where $A/J(A) = \prod_{i=1}^m M_{n_i}(D_i)$.
Then $$l(A) = \dim J(A)+\sum\limits_{i=1}^m n_i-1+n_i(n_i-1)[D_i:k]/2+ n_il(D_i) \hspace{0.5mm}.$$
\end{thm}

\noindent Some consequences of this theorem are listed below. The first one is immediate.

\begin{cor}\label{matrix} $l\big(M_n(k)\big) = n(n+1)/2 + n-1$.
\end{cor}

\noindent Note that the first summand on the RHS is $\dim T_n(k)$ and the second is $r$, the rank of $M_n(k)$.

\vspace{2mm}\noindent The next corollary is an analogue of Theorem $1$ of \cite{BLS2}, which gives a formula for the length of an algebraic group over an algebraically closed field.

\begin{cor}\label{lengthcor0} Let $A$ be an algebra and let $B$ be a Borel subalgebra of $A/J(A)$. Then $l(A)=\dim J(A)+l(B)+ r$, where $r$ is the rank of $A/J(A)$. If $A$ is split then $l(B)=\dim B$.
\end{cor}

\noindent If $A$ is a semisimple algebra then, by Corollary \ref{lengthcor0}, the length of $A$ is equal to its parabolic length. That is, there is an unrefinable chain of $A$ of maximal length which passes through a Borel subalgebra.

%\begin{cor}\label{lengthcor1} Let $A$ be an algebra. Let $N$ be a maximal nilpotent subalgebra of $A/J(A)$ and let $T$ be a maximal basic semisimple subalgebra of $A/J(A)$. Then $l(A) =\dim J(A)+ \dim N + l(T)+r$, where $r$ is the rank of $A/J(A)$.\end{cor}

\vspace{2mm}\noindent Combining Theorem \ref{length} with Galois theory of division algebras we deduce the following.

\begin{cor}\label{lengthcor} Let $A$ be a simple algebra, where $A=M_n(D)$ for some division algebra $D$ that is Galois over $k$. Then $l(A) \geq (\dim A)/2$ if and only if $[D:k] \in \{1,2,3,4,6,8\}$ or $[D:k]=5$ and $n>1$.
\end{cor}

\noindent
In particular, the inequality $l(A) \ge (\dim A)/2$ holds if $k$ is algebraically closed.

\vspace{2mm}\noindent Next, we discuss the depth and the chain difference of algebras.

% We define $e(k) := \sup\{[k':k] \mid k' \textnormal{ is a finite dimensional field extension of } k \}$.
\vspace{2mm}\noindent Let $\overline{k}$ denote the algebraic closure of $k$. If $[\overline{k}:k]$ is finite then either $k=\overline{k}$ or $k$ has characteristic $0$, $[\overline{k}:k]=2$ and $\overline{k} = k(i)$ with $i^2 = -1$; this is a result of Artin communicated to us by E. de Shalit, see Corollary VIII.$9.2$ of \cite{L}. Examples of fields $k$ that satisfy $[\overline{k}:k]=2$ are the real numbers and the real algebraic numbers.
% Such a field $k$ that satisfies $[\overline{k}:k]=2$ is called a \textit{real closed field}.

% Out of curiosity, is $\lambda(M_n(D)) \geq \lambda(M_n(k))$ true???
\begin{thm}\label{depth} Let $A=M_n(D)$ where $D$ is a division algebra. Then $\lambda(A) \leq 6 \log_2 n+\lambda(D)$. If $k$ is algebraically closed then $\lambda(A) \geq 3 \log_2 n+1$. If $[\overline{k}:k]=2$ then $\lambda(A) \geq \log_2 (\dim A)+1$.
\end{thm}

\noindent It is well known (see $\S \ref{preliminaries}$) that any division algebra $D$ contains a maximal subfield with degree $\sqrt{[D:Z(D)]}$ over $Z(D)$. So, in the notation of Theorem \ref{depth}, if $[\overline{k}:k]=2$ then $\sqrt{[D:k]} \leq 2$ and hence $2\log_2 n \leq \log_2 (\dim A) \leq 2 \log_2 n +2$.

\vspace{2mm}\noindent Combining the above result with Theorem \ref{length} we obtain the following.

\begin{cor}\label{depthcor} Let $A=M_n(D)$. Then $n \leq \sqrt{2\cd(A)+18}$.
\end{cor}

\noindent For any positive integer $n$, let $\Omega(n)$ be the number of prime divisors of $n$ (counting multiplicities). The following result gives an upper bound for the depth of a simple algebra over a field $k$ such that, for every $r \in \N$, there exists a field extension of $k$ of degree $r$. For example, all finite fields, $\mathfrak{p}$-adic fields and algebraic number fields satisfy this property. % For each of these fields the absolute Galois group has $\hat{Z}$ as a quotient, and since the absolute Galois group is the inverse limit of the Galois groups of finite field extensions, for every $r \in \N$ there exists a Galois field extension $L$ of $k$ with $\Gal(L/k) \cong \Z_r$.

\begin{thm}\label{depthfinitefields} Let $k$ be a field such that, for every $r \in \N$, there exists a field extension of $k$ of degree $r$. Let $A$ be a simple algebra, where $A=M_n(D)$. Then $\lambda(A) \leq \lambda(D)+\min\{2\Omega(n),15\}$. In particular, $\lambda\big(M_n(k)\big) \leq 16$.
\end{thm}

\noindent The proof of this result applies Helfgott's solution to the ternary Goldbach conjecture \cite{H} (which was applied for groups in \cite{BLS3}).

%\vspace{2mm}\noindent We now investigate the chain difference of algebras.

\vspace{2mm}\noindent Let $A$ be an algebra. We say that $A$ satisfies \textit{condition $(*)$} if $A/J(A) \cong M_2(k)^{\delta} \times \prod_{i=1}^m D_i$ for non-negative integers $\delta,m$ and division algebras $D_i$ such that
\begin{itemize}
\item $\cd(D_i)=0$ for each $1 \leq i \leq m$,
\item either $\delta=1$ and there does not exist a quadratic field extension of $k$ or $\delta=0$, and
\item any division subalgebra of $A/J(A)$ either embeds in $D_i$ for precisely one $i$ or is isomorphic to $k$.
\end{itemize} If $A$ is basic split (e.g. $A=T_n(k)$) then certainly $A$ satisfies condition $(*)$.

\vspace{2mm}\noindent Next, we study when the length of an algebra equals its dimension, and when its chain difference is zero.

\begin{thm}\label{dimthings} Let $A$ be an algebra. Then

$(i)$ $l(A)=\dim A$ if and only if $A/J(A)$ is a direct product of copies of $k$, $M_2(k)$ and quadratic field extensions of $k$.

$(ii)$ If $\cd(A)=0$ then $A$ satisfies condition $(*)$. Conversely, if $A$ is basic split or if $A$ is semisimple and satisfies condition $(*)$ then $\cd(A)=0$.
\end{thm}
% See p3 of length/depth of alg groups paper for a discussion on when the precise converse holds and when it fails. For example, if there is a M_2(k) in A/J(A) and A=A/J(A) \times J(A) then the converse holds, but if A/J(A) acts irreducibly on J(A) then it doesn't.

\noindent If $k$ is algebraically closed then it follows from Theorem \ref{dimthings} that $\lambda(A) = \dim A$ if and only if $\lambda(A) = l(A)$. However, this is not true for all fields. For example, the quaternions $\mathbb{H}$ is a $4$-dimensional $\R$-algebra with $\lambda(\mathbb{H})=l(\mathbb{H})=3$. If $A$ is basic split then $\lambda(A)=l(A) =\dim A$.

\vspace{2mm}\noindent The notions of chain difference was similarly defined for finite groups $G$.
It was shown by Iwasawa \cite{I} in 1941 that $\cd(G) = 0$ (namely, all unrefinable chains in $G$ have the same length) if and only if
$G$ is supersolvable. Combining this with the fundamental theorem of Galois theory of division algebras \cite{W}
(see Lemma \ref{multiplicativity} below), we immediately obtain the following.

\begin{cor}\label{division} Let $D$ be a division algebra which is Galois over $k$.
Then $\cd(D)=0$ if and only if $\Gal(D/k)$ is supersolvable.
\end{cor}

\noindent More results on the chain difference of finite groups and finite simple groups in particular can be found in \cite{BLS3}.
In particular, by Theorem 12 of \cite{BLS3}, if $G$ is a finite group, and $R(G)$ is the solvable radical of $G$, then
$|G/R(G)| \le 10 \cd(G)$. Thus finite groups with bounded chain difference are solvable-by-bounded.
We obtain some ring-theoretic analogues of this phenomenon.

\vspace{2mm}\noindent Any associative algebra $A$ is naturally a Lie algebra under the Lie bracket $[a, b] = ab - ba$ for $a,b \in A$. There exists a unique maximal Lie-solvable ideal of $A$, that we call the \textit{solvable radical} $R(A)$ of $A$. Clearly $R(A) \supseteq J(A)$. For $k$ algebraically closed we have $R(A)/J(A) \cong k^l$ for some $l \ge 0$ and $A/R(A) \cong \prod_{i=1}^m M_{n_i}(k)$ where $m \ge 0$ and $n_i \ge 2$.
%There does not always exist a maximal Lie-solvable subalgebra of $A$. However, if $A$ is split and semisimple then the Borel subalgebras of $A$ are precisely the maximal Lie-solvable subalgebras of $A$. % This justifies the nomenclature 'Borel subalgebra' of an associative algebra.

\begin{thm}\label{solvrad} Let $A$ be an algebra. Then $\dim A/R(A) \le 9 \cd(A)$ unless $A/R(A)\cong M_2(k)$ and there does not exist a quadratic field extension of $k$.
\end{thm}

\noindent Note that the converse of Theorem \ref{solvrad} does not always hold. For example, let $k$ be algebraically closed and let $A$ be the subalgebra of $M_3(k)$ consisting of all matrices with no non-zero entry in the bottom row. Then $A/R(A)\cong M_2(k)$ and $\cd(A)=1$. So indeed $\dim A/R(A) \le 9 \cd(A)$.

\vspace{2mm}\noindent It is natural to ask whether $\dim A$ is bounded above in terms of $l(A)$, or even in terms of $\lambda(A)$.
It turns out that the answers to these questions are negative in general, but positive over certain fields.

\begin{thm}\label{ld}
(i) Let $k = \Q$. Then for every positive integer $N$ there exists a $k$-algebra $A$ satisfying $l(A) = 2$
and $\dim A > N$.

$(ii)$ Let $k$ be a field such that $[\overline{k}:k] < \infty$. Then there exists a function $f: \N \to \N$ such that, for every $k$-algebra $A$, $\dim A \le f(\lambda(A))$. In particular, this holds when $k$ is algebraically closed, or $\R$, or the real algebraic numbers.
\end{thm}

\section{Preliminaries}\label{preliminaries}

\noindent Let $k$ be any field.

\vspace{2mm}\noindent Let $A$ be a simple algebra (over $k$). By Wedderburn's Theorem, we can write $A=M_n(D)$ for some positive integer $n$ and division algebra $D$.

\vspace{2mm}\noindent Some remarks on notation. Let $\alpha=(\alpha_1,...,\alpha_r)$ be a partition of $n$ (i.e. $n=\sum_{i=1}^r\alpha_i$ where $\alpha_i$ are positive integers) and suppose $r \geq 2$. Let $P_{\alpha}(D)$ (resp. $L_{\alpha}(D)$) be the subalgebra of $A$ that consists of all block upper triangular (resp. block diagonal) matrices with $r$ blocks on the diagonal such that the $i$'th block has size $\alpha_i$. % We call $P$ a \textit{standard parabolic subalgebra} of $A$.  A \textit{parabolic subalgebra} of $A$ is any conjugate of a standard parabolic subalgebra of $A$.
Observe that $J\big(P_{\alpha}(D)\big)$ is the subalgebra of $A$ that consists of all block strictly upper triangular matrices with $r$ blocks on the diagonal such that the $i$'th block has size $\alpha_i$. % This paper https://dml.cz/bitstream/handle/10338.dmlcz/140575/CzechMathJ_60-2010-2_6.pdf tells us how to find the ideals, then take the intersection of all of them. E.g. if l=2 there are 3 proper ideals, (1) the upper two blocks, (2) the right two blocks, (3) the upper right block
So we can decompose $P_{\alpha}(D)=L_{\alpha}(D) \oplus J\big(P_{\alpha}(D)\big)$ as a direct sum of vector spaces. % [Note that this does not just follow from the Wedderburn-Malcev Principal Theorem \textbf{Aner says: it might be good to state it and give a reference} as $L$ is not necessarily a separable algebra, for example if $Z(D)$ is an inseparable field extension of $k$].
Any parabolic subalgebra of $A$ is conjugate to $P_{\alpha}(D)$ for some $\alpha$. If $r=n$ then $P_{\alpha}(D)=T_n(D)$ and $J\big(P_{\alpha}(D)\big)=U_n(D)$.

\vspace{2mm}\noindent If $r=2$ then the chain $A > P_{\alpha}(D)> L_{\alpha}(D)$ is unrefinable. % Proposition 2.2 of https://arxiv.org/pdf/1403.0773.pdf
We will usually denote this chain by $A > P_{\alpha}(D)> \prod_{i=1}^r M_{\alpha_i}(D)$. All subalgebras of $P_{\alpha}(D)$ that are isomorphic to $\prod_{i=1}^r M_{\alpha_i}(D)$ are conjugate to $L_{\alpha}(D)$, and so there is no ambiguity with this notation.

\vspace{2mm}\noindent We say that $A$ is \textit{central} if $Z(A) \cong k$. For any subalgebra $B$ of $A$, let $C_A(B)$ denote the centraliser of $B$ in $A$.

\begin{theorem}[p.$53$ of \cite{A}]\label{simplesubalgebras} Let $A$ be a central simple algebra. Let $B$ be a simple subalgebra of $A$ that contains the identity of $A$. Then $C_A(B)$ is a simple subalgebra of $A$, $C_A\big(C_A(B)\big)=B$ and $[A:k]=[B:k]\cdot[C_A(B):k]$.
\end{theorem}
% Also see Proposition 7.7.7 of https://math.dartmouth.edu/~jvoight/quat-book.pdf, Theorem 6.2 of http://math.uga.edu/~pete/noncommutativealgebra.pdf or https://stacks.math.columbia.edu/tag/074S

% I used to think that $\Delta=F$ (it appears to say this in p286 of https://link.springer.com/content/pdf/10.1007/BF02762016.pdf and also possibly in Albert). But this cannot be true (e.g. consider the case n=1).
\begin{corollary}\label{simplesubalgebrascor} Let $A$ be a central simple algebra, where $A=M_n(D)$. Let $F \subseteq A$ be a field extension of $k$. Then $C_A(F) \cong M_t(\Delta)$ for some division algebra $\Delta$ and positive integer $t$ such that $Z(\Delta) \cong F$ and $n^2/t^2=[F:k]^2 [\Delta:F] \big/[D:k]$. Moreover, $C_A(F)=F$ if and only if $F$ is a maximal commutative subalgebra of $A$ if and only if $[F:k]=n\sqrt{[D:k]}$.

\begin{proof} By Theorem \ref{simplesubalgebras}, $C_A(F)$ is a simple subalgebra of $A$ that satisfies $F=C_A\big(C_A(F)\big)=Z\big(C_A(F)\big)$. So $C_A(F)\cong M_t(\Delta)$ for some positive integer $t$ and division algebra $\Delta$ by Wedderburn's Theorem. Then $$n^2 [D:k]=[F:k] \cdot t^2 [\Delta:k]$$ again by Theorem \ref{simplesubalgebras}. The final statement follows immediately from Theorem \ref{simplesubalgebras}.
\end{proof}
\end{corollary}

\noindent A subfield of a simple algebra is called \textit{strictly maximal} if it is self-centralising.

\vspace{2mm}\noindent Our main tool is the following result by Iovanov and Sistko which classifies maximal subalgebras of a simple algebra into three families.

\begin{theorem}[\textit{Lemma $3.6$ of \cite{IS}}]\label{classthm1} Let $A$ be a simple algebra, where $A=M_n(D)$. A subalgebra $B$ of $A$ is maximal if and only if it is of the following forms:

\noindent $(S1)$: $B$ is a maximal parabolic subalgebra of $A$

\noindent or $B$ is a simple subalgebra of $A$ (say $B \cong M_t(\Delta)$ for some division algebra $\Delta$, $L:=Z(A) \cong Z(D)$ and $F:=Z(B) \cong Z(\Delta)$) such that $\dim B$ divides $\dim A$ and either

\noindent $(S2)$: $F \supset L$ is a minimal field extension and $B=C_A(F)$, or

\noindent $(S3)$: $L \supset F$ is a minimal field extension, $t$ divides $n$ and $M_{n/t}(D) \cong L \otimes_F \Delta$.
\end{theorem}

\noindent If $k$ is algebraically closed then it follows from Theorem \ref{classthm1} that there are no non-trivial proper irreducible subalgebras of $M_n(k)$. This is a well-known theorem of Burnside's.

\begin{corollary}\label{classthm1cor} Let $A$ be a simple algebra and let $B$ be a maximal subalgebra of $A$. If $B$ is of type $(S2)$ (resp. $(S3)$) then $\dim A=m\dim B$ where $m=[Z(B):Z(A)]$ (resp. $m=[Z(A):Z(B)]$).
\begin{proof} Denote $A=M_n(D)$, $B = M_t(\Delta)$, $L:=Z(A)$ and $F:=Z(B)$.

\vspace{2mm}\noindent If $B$ is of type $(S2)$ then $$[A:k]=[A:L][L:k]=[F:L][C_A(F):L][L:k]=[F:L] [B:k]$$ by Theorems \ref{simplesubalgebras} and \ref{classthm1}.

\vspace{2mm}\noindent Now let $B$ be of type $(S3)$. Then $M_{n/t}(D) \cong L \otimes_F \Delta$ by Theorem \ref{classthm1} and so $$[A:F]= n^2[D:F]=[L:F]t^2[\Delta:F]=[L:F][B:F].$$ Hence $[A:k]=[A:F][F:k]=[L:F][B:k]$.
\end{proof}
\end{corollary}

\noindent We are grateful to Iovanov and Sistko for communicating to us the proof of the following lemma.

\begin{lemma}\label{tleqn} Let $A$ be a simple algebra, where $A=M_n(D)$. Let $B$ be a simple subalgebra of $A$, where $B=M_t(\Delta)$. Then $t \leq n$.
\begin{proof} For any algebra $S$, let $N(S)$ denote the largest possible nilpotency index of any nilpotent element in $S$. Let $n \in A$ be nilpotent. As in the case for fields, we can choose a basis for the left $A$-module $D^n$ such that $n$ is an upper triangular matrix. It follows that $N(A)=n$. So we have $t=N(B) \leq N(A)=n$.
\end{proof}
\end{lemma}

\noindent Henceforth (unless otherwise stated) let $A$ be any algebra. By Wedderburn's Theorem, we can write $A/J(A) = \prod_{i=1}^m M_{n_i}(D_i)$ for some positive integers $m$, $n_1$, ..., $n_m$ and division algebras $D_1$,..., $D_m$. Let $\pi:A \to A/J(A)$ be the natural projection.

\begin{theorem}[Theorems $2.5$ and $3.10$ of \cite{IS}]\label{classthm2} A subalgebra $B$ of $A$ is maximal if and only if it is of the following forms:

\noindent $(i)$: $\pi(B)=A/J(A)$, $J(B)=B \cap J(A) \subsetneq J(A)$, there exists an ideal $I \subset B \cap J(A)$ of $A$ such that $J(A)/I$ is a simple $B$-bimodule and $A/I = B/I \oplus J(A)/I$,

\noindent $(ii)$: $\pi(B)$ is $A/J(A)$-conjugate to $\Delta^2(n_i,D_i) \times \prod_{l \neq i,j} M_{n_l}(D_l)$ for some $1 \leq i \neq j \leq m$ where $n_i=n_j$, $D_i \cong D_j$ and $\Delta^2(n_i,D_i)$ is the image of the diagonal embedding $M_{n_i}(D_i) \to M_{n_i}(D_i) \times M_{n_j}(D_j)$, or

\noindent $(iii)$: $\pi(B)$ is $A/J(A)$-conjugate to $B_j \times \prod_{j \neq i} M_{n_j}(D_j)$ for some $1 \leq j \leq m$ where $B_j$ is a maximal subalgebra of $M_{n_j}(D_j)$ of type $(S1)$, $(S2)$ or $(S3)$.
\end{theorem}

\noindent A maximal subalgebra $B$ of $A$ that satisfies $\pi(B)=A/J(A)$ is said to be of \textit{split type}. Otherwise, $B$ is of \textit{semisimple} type. By Lemma $2.3$ of \cite{IS}, $B$ is of semisimple type if and only if $J(A) \subseteq B$. By Corollary $3.12$ of \cite{IS}, maximal subalgebras of $A$ of semisimple type are in $1-1$ correspondence with maximal subalgebras of $A/J(A)$.

\begin{lemma}\label{subalgideal} Every maximal subalgebra of $A$ is an ideal if and only if $A=k$ or $A$ is nilpotent.
\end{lemma}

\begin{proof} Write $A/J(A) = \prod_i M_{n_i}(D_i)$ where $D_i$ are division algebras. The assumption on $A$ is inherited by quotients, hence every maximal subalgebra of the simple algebras $M_{n_i}(D_i)$
is an ideal, which must be $0$. This easily implies $n_i=1$ and $D_i = k$ for all $i$. Thus $A/J(A) = k^n$ for some $n \ge 0$, and this quotient satisfies the assumption only if $n=0,1$, namely $A/J(A)$ is $0$ or $k$.
\end{proof}

\begin{lemma}\label{ideallength} Let $I$ be an ideal of $A$. Then $l(A)=l(I)+l(A/I)$.
\begin{proof} Let $Q:=A/I$ and let $\rho:A \to Q$ be the natural projection. Let $I>I_1>...>0$ be an unrefinable chain of maximal length of $I$ and let $Q>Q_1>...>0$ be an unrefinable chain of maximal length of $Q$. Then $A=\rho^{-1}(Q)>\rho^{-1}(Q_1)>...>I>I_1>...>0$ is an unrefinable chain of $A$ of length $l(I)+l(Q)$.

\vspace{2mm}\noindent Conversely, let $A=A_0>A_1>...>A_t=1$ be an unrefinable chain of maximal length of $A$. Let $i \in \{0,1,...,t-1 \}$. It is not possible that both $A_i+I=A_{i+1}+I$ and $A_i \cap I = A_{i+1}\cap I$. So the derived chains $I=A_0 \cap I> A_1 \cap I >...>0$ and $Q=(A_0+I)/I>(A_1+I)/I>...>0$ have lengths $l_1$ and $l_2$ respectively with $l_1+l_2 \geq t$.
\end{proof}
\end{lemma}
% Corollary \ref{radicalcor}, it suffices to consider the case where $A$ is semisimple (since $J(A)$ is an ideal of $A$). That is, $A=\prod_{i=1}^m M_{n_i}(D_i)$.
% \vspace{2mm}\noindent For each $i$, let $M_{n_i}(k)=X^i_0>X^i_1>....>X^i_{t_i}=1$ be an unrefinable chain of minimal length $t_i$ of $M_{n_i}(k)$. Observe that the chain $$A > X^1_1 \times \prod_{i >1} M_{n_i}(k)>X^1_2 \times \prod_{i >1} M_{n_i}(k)>...>\prod_{i >1} M_{n_i}(k)>X^2_1 \times \prod_{i >2} M_{n_i}(k)>...>1$$ has length $\sum_{i=1}^m t_i$. Conversely, let $A=A_0>A_1>...>A_t=1$ be an unrefinable chain of length $t$ in $A$. For each $i$, consider the induced chain $A \cap M_{n_i}(k) \geq A_1 \cap M_{n_i}(k) \geq ...\geq 1$. Let $l_i$ be the length of the associated unrefinable chain. For every $0 \leq j < t$, there exists an integer $i$ such that $A_j \cap M_{n_i}(k) \neq A_{j+1} \cap M_{n_i}(k)$ (as otherwise $A_j =A_{j+1}$). Hence $\sum_{i=1}^m l_i \geq t$.

\noindent Recall that the Frattini subalgebra $F(A)$ of an algebra $A$ is the intersection
of all maximal subalgebras of $A$. It is easy to see that, for a subalgebra $B \le A$, if $B + F(A) = A$
then $B = A$. We need the following result of Towers, see Theorem 6 of \cite{To}.

\begin{lemma}\label{nilp} Let $A$ be a nilpotent algebra. Then $F(A)=A^2$.
\end{lemma}

\noindent We can now obtain the following result.

\begin{lemma}\label{nilpotent} Let $A$ be a nilpotent algebra. Then $l(A)=\lambda(A)=\dim A$.
\begin{proof} It suffices to show that any maximal subalgebra of $A$ has codimension $1$.
Let $B < A$ be a maximal subalgebra. Then, by Lemma \ref{nilp} and the remark preceding it
we have $B+A^2 = B + F(A) < A$. It follows by the maximality of $B$ that $B+A^2 = B$, namely
$A^2 \le B$. Hence $B$ is an ideal of $A$, and $xy=0$ for all $x, y \in A/B$.
Therefore every subspace between $B$ and $A$ is a subalgebra. Since $B$ is maximal we must have $\dim A/B = 1$,
as required.
\end{proof}
\end{lemma}

\begin{corollary}[\textit{additivity of length}]\label{reduction} $l(A)=\dim J(A)+\sum_{i=1}^m l\big(M_{n_i}(D_i)\big)$.
\begin{proof} By Lemmas \ref{ideallength} and \ref{nilpotent}, $l(A)=l\big(J(A)\big)+l(A/J(A))=\dim J(A)+l(A/J(A))$ since $J(A)$ is a nilpotent ideal of $A$. The result then follows from Lemma \ref{ideallength} since $M_{n_i}(D_i)$ is an ideal of $A/J(A)$ for each $i$.
\end{proof}
\end{corollary}

\noindent For any finite group $G$, let $l(G)$ (resp. $\lambda(G)$) denote the maximal (resp. minimal) length of an unrefinable chain of subgroups of $G$.

\begin{lemma}\label{multiplicativity} Let $D$ be a finite dimensional division algebra. Then $l(D) \leq \Omega\big([D:k]\big)+1$. If $D$ is Galois over $k$, with associated Galois group $\Gamma:=\Gal(D/k)$, then $l(D) = l(\Gamma)+1$ and $\lambda(D) = \lambda(\Gamma)+1$. If $D$ is Galois over $k$ and $\Gamma$ is solvable then $l(D) = \Omega\big([D:k]\big)+1$.
\begin{proof} We first note that, since $\dim D < \infty$ as a $k$-algebra, every subalgebra $L \ne 0$ of $D$ is a division algebra.
Indeed, any $a \in L$ is algebraic over $D$, so if, in addition, $a \ne 0$, then $a^{-1} \in D$ can be expressed
as a polynomial in $a$ with coefficients in $k$, hence $a^{-1} \in L$.

\vspace{2mm}\noindent Let $$D =D_0> D_1 > ... >D_t=k>D_{t+1}=0$$ be an unrefinable chain of $D$. Let $i \leq t$ be a positive integer. By Theorem \ref{classthm1}, $D_i$ is a maximal subalgebra of $D_{i-1}$ of type $(S2)$ or $(S3)$ and so $[D_i:k]$ divides $[D_{i-1}:k]$. Then $t \leq \Omega\big([D:k]\big)$ by induction on $i$.

\vspace{2mm}\noindent Henceforth let $D$ be Galois over $k$ and let $\Gamma:=\Gal(D/k)$. There is a $1-1$ correspondence between subgroups of $\Gamma$ and subalgebras of $D$ that contain $k$ (the fundamental theorem of Galois theory of division algebras, Theorem $8$ of \cite{W}). That is, $D =D_0> D_1 > ... >D_t=k>D_{t+1}=0$ gives rise to an unrefinable chain
$\Gamma=\Gamma_0 > \Gamma_1 > \ldots > \Gamma_t=1$ in $G$, where $\Gamma_i = \Gal(D/D_{t-i})$ ($i = 0, \ldots , t)$. Conversely, any unrefinable chain $\Gamma=\Gamma'_0 > \Gamma'_1 > \ldots > \Gamma'_t=1$ in $\Gamma$ gives rise to
an unrefinable chain of subalgebras $D = D'_0 > D'_1 > \ldots > D'_t = k > D'_{t+1} = 1$ where, for $i=0, \ldots , t$,
$D'_i = D^{\Gamma'_{t-i}}$, the subfield of common fixed points of $\Gamma'_{t-i}$. Hence $l(D) = l(\Gamma)+1$ and $\lambda(D) = \lambda(\Gamma)+1$.

\vspace{2mm}\noindent If $\Gamma$ is solvable then $l(\Gamma)=\Omega(|\Gamma|)$ (this is true of all solvable groups). This proves the final assertion since $|\Gamma|=[D:k]$.
\end{proof}
\end{lemma}

\begin{lemma}\label{strictmax} Let $A=M_n(D)$ be simple and contain a strictly maximal subfield. Then $\lambda(A) \leq 2\Omega(n)+\Omega\big([D:k]\big)+1$.
\begin{proof} Let $F$ be a strictly maximal subfield of $A$. Then $[F:Z(D)]=n\sqrt{[D:Z(D)]}$ by Corollary \ref{simplesubalgebrascor}. Let $F > F_1 > ... >F_t=1$ be an unrefinable chain of $F$ such that $F_j =Z(D)$ for some $j<t$. The chain $$ A =C_A(F_j)> C_A(F_{j-1}) >...> C_A(F)=F > F_1 > ... > 0$$ is unrefinable by Theorem \ref{classthm1} and has length at most $$2\Omega\big([F:Z(D)]\big)+\Omega\big([Z(D):k]\big)+1=2\Omega(n)+\Omega\big([D:k]\big)+1$$ by Lemma \ref{multiplicativity}.
\end{proof}
\end{lemma}

% non-trivial means proper and non-zero
\begin{lemma}\label{depthlowerbound} Let $I$ be a non-trivial ideal of $A$. Then $\lambda(A) \geq \lambda(A/I)+1$.
\begin{proof} Let $\rho:A \to A/I$ be the natural projection. We first observe that $\lambda(A) \geq \lambda(A/I)$ by taking the preimage under $\rho$ of any unrefinable chain of $A/I$. Let \begin{equation}\label{hakal} A=A_0>A_1>...>A_t=1 \end{equation} be an unrefinable chain of $A$ of minimal length $t$. Let $M_i:=(A_i+I)/I$ for each $i \in \{0,1,...,t\}$.

\vspace{2mm}\noindent Assume (for a contradiction) that $\lambda(A) = \lambda(A/I)$. Then $M_i \supsetneq M_{i+1}$ for each $i$. So $A_i+I \supsetneq A_{i+1}+I$ for each $i$. Let $j$ be maximal such that $I \subseteq A_j$. Then $A_j \supsetneq A_{j+1}+I \supsetneq A_{j+1}$, contradicting the unrefinability of $(\ref{hakal})$. So $\lambda(A) \geq \lambda(A/I)+1$.
\end{proof}
\end{lemma}

\begin{corollary}\label{depthlowerboundcor} Let $A$ be an algebra, where $A/J(A) = \prod_{i=1}^m M_{n_i}(D_i)$. Then $\lambda(A) \geq \smash{\max\limits_{i=1,...,m}} \big\{\lambda\big(M_{n_i}(D_i)\big) \big\}+m-\delta$, where $\delta=1$ if $A$ is semisimple and $0$ otherwise.
\begin{proof} We first consider the case where $A$ is not semisimple. Then $J(A)$ is a non-trivial ideal of $A$. So $\lambda(A) \geq \lambda\big(A/J(A)\big)+1$ by Lemma \ref{depthlowerbound}.

\vspace{2mm}\noindent Henceforth assume that $A$ is semisimple. We show that $$\lambda(A) \geq \smash{\max\limits_{i=1,...,m}} \big\{\lambda\big(M_{n_i}(D_i)\big) \big\}+m-1$$ by induction on $m$. If $m=1$ then we are done. So let $m>1$. Let $i_0$ satisfy $\lambda\big(M_{n_{i_0}}(D_{i_0})\big) \geq \lambda\big(M_{n_i}(D_i)\big)$ for all $1 \leq i \leq m$. Take any $j \neq i_0$ and note that $M_{n_j}(D_j)$ is an ideal of $A$. Then $$\lambda(A) \geq \lambda\big(A \big/ M_{n_j}(D_j)\big)+1 \geq \smash{\max\limits_{i=1,...,m}} \big\{\lambda\big(M_{n_i}(D_i)\big) \big\}+m-1$$ by Lemma \ref{depthlowerbound} and the inductive hypothesis.
\end{proof}
\end{corollary}
% \vspace{2mm}\noindent If $\lambda(A)=1$ then $A \cong k$ and the claim is true. Let $B$ be a maximal subalgebra of $A$. Since $A$ is semisimple, $B$ is either of type $(ii)$ or $(iii)$ (in the language of Theorem \ref{classthm2}). For either of these types, it is easy to see that $$\lambda(B) \geq \smash{\max\limits_{i=1,...,m}} \big\{\lambda\big(M_{n_i}(D_i)\big) \big\}+m-2$$ by the inductive hypothesis.
% \vspace{2mm}\noindent If $B$ is of split type then $B/J(B)=B/(B \cap J(A)) \cong (B+J(A))/J(A)=A/J(A)$ by Theorem \ref{classthm2} and so $\lambda(B) \geq \lambda\big(B/J(B)\big)=\lambda\big(A/J(A)\big)$. If $B$ is of semisimple type then $B/J(A)$ is a maximal subalgebra of $A/J(A)$ and $B$ is not semisimple.

\section{Proof of Theorem \ref{finlen}}\label{prooffinlen}

\noindent Since $\lambda(A) \le l(A) \le \dim A$, it suffices to show
that $\lambda(A) < \infty$ implies $\dim A < \infty$. Let
\[
A = A_0 > A_1 > \ldots > A_t =0
\]
be an unrefinable chain of length $t = \lambda(A)$. We prove by induction on $i$ ($0 \le i \le t$) that
$\dim A_{t-i} < \infty$, the case $i=0$ being obvious. Suppose now that $\dim A_{t-i} < \infty$. Then
$A_{t-i}$ is a finite dimensional maximal subalgebra of $A_{t-(i+1)}$.

\vspace{2mm}\noindent The main result of \cite{LL} shows that an associative algebra with a finite dimensional maximal subalgebra
is finite dimensional. It follows that $\dim A_{t-(i+1)} < \infty$, completing the inductive step.
For $i=t$ we obtain $\dim A = \dim A_0 < \infty$, completing the proof. We are grateful to Agata Smoktunowicz for providing
the above reference.

\section{Proof of Theorem \ref{length}}\label{prooflength}

\noindent It suffices to consider the case where $A$ is simple by Corollary \ref{reduction}. So let $A=M_n(D)$ for some positive integer $n$ and division algebra $D$.

\vspace{2mm}\noindent Denote $\Xi_A := n-1+n(n-1)[D:k]/2+ nl(D)$. We will show that $l\big(M_n(D)\big)=\Xi_A$. The case $n=1$ is easy. So we assume that $n>1$.

\begin{lemma}\label{lowerbound} $l(A) \geq \Xi_A$.
\begin{proof} Recall that $T_n(D)$ (resp. $U_n(D)$) denotes the subalgebra of $A$ consisting of all upper (resp. strictly upper) triangular matrices. Observe that $J\big(T_n(D)\big)=U_n(D)$ and $T_n(D)/U_n(D) \cong D^n$. So $l\big(T_n(D)\big)=\dim U_n(D) + nl(D)=n(n-1)[D:k]/2+nl(D)$ by Corollary \ref{reduction}.

 \vspace{2mm}\noindent The chain of parabolic subalgebras $$A > P_{n-1,1}(D) > P_{n-2,1,1}(D) > ... > P_{1,1,...,1}(D)=T_n(D)$$ is unrefinable of length $n-1$. Hence $l(A) \geq \Xi_A$.
\end{proof}
\end{lemma}

\noindent It remains to show that $l(A) \leq \Xi_A$. We prove this using induction on $l(A)$. Let $M$ be a maximal subalgebra of $A$. By Theorem \ref{classthm1}, $M$ is of type $(S1)$, $(S2)$ or $(S3)$. We show that $l(M) < \Xi_A$ for each of these possibilities.

\vspace{2mm}\noindent \underline{Type $(S1)$}:

 \vspace{2mm}\noindent Let $M$ be a maximal parabolic subalgebra of $A$. That is, $M$ is conjugate to $P_{r,n-r}(D)$ for some positive integer $r <n$. Observe that $\dim J(M)=r(n-r)[D:k]$ and so \begin{align*}
l(M)&=\dim J(M)+l\big(M_r(D)\big)+l\big(M_{n-r}(D)\big) \\
&=r(n-r)[D:k]+n-2+(n^2+2r^2-2rn-n)[D:k]/2+ nl(D) \\
&= \Xi_A -1
\end{align*} by Corollary \ref{reduction} and the inductive hypothesis.

\vspace{2mm}\noindent \underline{Type $(S2)$ or $(S3)$}:

\vspace{2mm}\noindent Let $M$ be of type $(S2)$ or $(S3)$. In particular, $M$ is simple and $\dim M$ divides $\dim A$. We write $M \cong M_t(\Delta)$ for some integer $t$ and division algebra $\Delta$. Then $\dim A=n^2[D:k]=mt^2[\Delta:k]=m\dim M$ for some integer $m \geq 2$. Observe that $t \leq n$ by Lemma \ref{tleqn}.

\vspace{2mm}\noindent We consider several possibilities for the parameters involved. If $t=1$ then $$l(M) = l(\Delta)\leq \Omega([\Delta:k])+1 \leq \Omega(n^2[D:k]/2)+1 \leq n+[D:k] < \Xi_A $$ by Lemma \ref{multiplicativity} and since we assumed that $n \neq 1$.

\vspace{2mm}\noindent In this paragraph we assume that $n > 2$. If $t=2$ then $$l(M)=1+[\Delta:k]+2l(\Delta) \leq 3+[\Delta:k]+2\Omega([\Delta:k]) \leq 3+2[\Delta:k] \leq 3+n^2[D:k]/4<\Xi_A$$ by the inductive hypothesis and Lemma \ref{multiplicativity}. If $t=3, 4$ or $5$ then \begin{align*}
 l(M)&= t-1+t(t-1)[\Delta:k]/2+tl(\Delta) \\
 &\leq t-1+t(t+1)[\Delta:k]/2 \\
& \leq t-1+t(t+1)n^2[D:k]/(4t^2) \\
&< 4 + n^2[D:k]/3 \\
&\leq \Xi_A
\end{align*} by the inductive hypothesis. If $t > 5$ then \begin{align*}
l(M)&=t-1+t(t-1)[\Delta:k]/2+tl(\Delta) \\
& \leq t\big(\Omega([\Delta:k])+2\big)+t(t-1)[\Delta:k]/2 -1\\
& \leq 3t[\Delta:k]/2+t(t-1)[\Delta:k]/2 -1 \\
& = t(t+2)[\Delta:k]/2-1 \\
& \leq 2t^2[\Delta:k]/3-1 \\
& \leq n^2[D:k]/3-1 \\
& < \Xi_A
\end{align*} by the inductive hypothesis, Lemma \ref{multiplicativity} and since $[\Delta:k] \neq 1$.

\vspace{2mm}\noindent It remains to consider the case where $n =2$ and $t=2$ (since $t \leq n$). We have $$ l(M)=1+[\Delta:k]+2l(\Delta) \leq 3+[\Delta:k]+2\Omega([\Delta:k]) \leq 3+2[\Delta:k] \leq 3+[D:k] < \Xi_A$$ by the inductive hypothesis, Lemma \ref{multiplicativity} and since $D \neq k$ (as otherwise $[\Delta:k]=1/m$, a contradiction).

\vspace{2mm}\noindent This completes the proof of Theorem \ref{length}.

%\section{Proof of Corollary \ref{lengthcor1}}\label{prooflengthcor1}

%\noindent Let $A$ be an algebra. We write $A/J(A)= \prod_{i=1}^m M_{n_i}(D_i)$. Let $N$ be a maximal nilpotent subalgebra of $A/J(A)$ and let $T$ be a maximal basic semisimple subalgebra of $A/J(A)$. Let $i \in \{1,...,m\}$. Let $A_i:=M_{n_i}(D_i)$ and let $\pi_i:A/J(A) \to A_i$ be the natural projection. Observe that $\pi_i(N)$ is a maximal nilpotent subalgebra of $A_i$ and $\pi_i(T)$ is a maximal basic semisimple subalgebra of $A_i$.

%\vspace{2mm}\noindent It is easy to see that $\pi_i(T)$ is conjugate to the subalgebra of diagonal matrices of $A_i$. By Theorem $2$ of \cite{SP}, $\pi_i(N)$ is conjugate to the subalgebra of strictly upper triangular matrices of $A_i$. Hence $$\dim N = \sum\limits_{i=1}^m \dim \pi_i(N) =\sum\limits_{i=1}^m n_i(n_i-1)[D_i:k]/2$$ and $$l(T)=\sum\limits_{i=1}^m l\big(\pi_i(T)\big) =\sum\limits_{i=1}^m n_il(D_i)$$ using Corollary \ref{reduction}. The result then follows from Theorem \ref{length}.

\section{Proof of Corollary \ref{lengthcor0}}\label{prooflengthcor0}

\noindent Write $A/J(A) = \prod_{i=1}^m M_{n_i}(D_i)$. By definition, $B$ is conjugate to the subalgebra $\prod_{i=1}^m T_{n_i}(D_i)$ of $A/J(A)$. So $J(B) \cong \prod_{i=1}^m U_{n_i}(D_i)$ and $B/J(B) \cong \prod_{i=1}^m (D_i)^{n_i}$. Then $$l(B)=\dim J(B)+\sum_{i=1}^m n_il(D_i)=\sum_{i=1}^m n_i(n_i-1)[D_i:k]/2+ n_il(D_i)$$ by Theorem \ref{length} and so $l(A)=\dim J(A)+l(B)+ r$.

\vspace{2mm}\noindent If $A$ is split then $l(B)=\sum_{i=1}^m n_i(n_i+1)/2=\dim B$.

\section{Proof of Corollary \ref{lengthcor}}\label{prooflengthcor}

\noindent Let $A=M_n(D)$ for some division algebra $D$ that is Galois over $k$. Observe that $$l(A)-(\dim A)/2 = n\big(1+l(D)-[D:k]/2\big)-1$$ by Theorem \ref{length}.

\vspace{2mm}\noindent Let $[D:k] \leq 6$ or $[D:k]=8$ (except for the case where $[D:k]=5$ and $n=1$). Then $\Gal(D/k)$ is solvable (as all finite groups of order less than $60$ are solvable). Hence $l(D) = \Omega\big([D:k]\big)+1$ by Lemma \ref{multiplicativity}. Then we have $$l(A)-(\dim A)/2 = n\big(2+\Omega([D:k])-[D:k]/2\big)-1 \geq 0.$$

\noindent We now prove the converse. If $[D:k] \geq 9$ or $[D:k]=7$ then $$1+l(D)-[D:k]/2 \leq 2+ \Omega([D:k])-[D:k]/2<0$$ using Lemma \ref{multiplicativity} and so $l(A)-(\dim A)/2 <0$. Finally, if $[D:k] = 5$ and $n=1$ then $l(A)-(\dim A)/2=-1/2<0$.

\section{Proof of Theorem \ref{depth}}\label{proofdepth}

\noindent Let $A=M_n(D)$. Let $D=D_0>D_1>...>k>1$ be an unrefinable chain of minimal length of division algebras. The chain $$ A > M_n(D_1)>...>M_n(k)$$ is unrefinable of length $\lambda(D)-1$. So $\lambda(A) \leq \lambda(D)-1 + \lambda\big(M_n(k)\big)$.

\vspace{2mm}\noindent Let $a$ be the unique integer satisfying $2^a \leq n <2^{a+1}$. Note that $a \leq \log_2 n$. We show that $\lambda\big(M_n(k)\big) \leq 6a+1$ by induction.

\vspace{2mm}\noindent If $a=0$ then $n=1$ and $\lambda\big(M_n(k)\big)=1$. If $n$ is even then the chain $$M_n(k) > P_{n/2,n/2}(k)>\big(M_{n/2}(k)\big)^2>M_{n/2}(k)$$ is unrefinable of length $3$. If $n$ is odd then the chain $M_n(k) >P_{n-1,1}(k)>M_{n-1}(k) \times k >... $ $$...\hspace{0.8mm}M_{n-1}(k)> P_{(n-1)/2,(n-1)/2}(k)>\big(M_{(n-1)/2}(k)\big)^2>M_{(n-1)/2}(k)$$ is unrefinable of length $6$. Observe that $2^{a-1} \leq \lfloor n/2 \rfloor <2^{a}$. So $\lambda\big(M_{\lfloor n/2 \rfloor}(k)\big) \leq 6(a-1)+1$ by the inductive hypothesis. Hence $\lambda\big(M_n(k)\big) \leq 6 + \lambda\big(M_{\lfloor n/2 \rfloor}(k)\big) \leq 6a+1$.

\vspace{2mm}\noindent In summary, we have $\lambda(A) \leq \lambda(D)-1 + 6a+1 \leq \lambda(D) + 6\log_2 n$.

\vspace{2mm}\noindent Now we assume that $k$ is algebraically closed. So $D=k$. We show that $\lambda(A) \geq 3\log_2 n+1$ by induction on $\lambda(A)$. If $\lambda(A)=1$ then $n=1$ and we are done.

\vspace{2mm}\noindent Let $B$ be a maximal subalgebra of $A$. By Theorem \ref{classthm1}, $B$ is conjugate to $P_{r,n-r}(k)$ for some integer $r$ satisfying $n/2 \leq r < n$. Rearranging, we have $\log_2 r \geq \log_2 n-1$. Then $$\lambda(B) \geq 2+\lambda\big(M_r(k)\big) \geq 2+3\log_2 r+1 \geq 3\log_2 n$$ by Corollary \ref{depthlowerboundcor} and the inductive hypothesis. This proves the second assertion of the theorem.

\vspace{2mm}\noindent Finally, we assume that $[\overline{k}:k]=2$. We show that $\lambda(A) \geq \log_2(\dim A)+1$ by induction on $\lambda(A)$. If $\lambda(A)=1$ then $A \cong k$ and certainly this inequality holds.

\vspace{2mm}\noindent Let $B$ be a maximal subalgebra of $A$. By Theorem \ref{classthm1}, $B$ is of type $(S1)$, $(S2)$ or $(S3)$. We show that $\lambda(B) \geq \log_2(\dim A)$ for each of these three possibilities.

\vspace{2mm}\noindent Let $B$ be of type $(S1)$. That is, $B$ is conjugate to $P_{r,n-r}(D)$ for some integer $r$ satisfying $n/2 \leq r <n$. Note that $B/J(B) \cong M_r(D) \times M_{n-r}(D)$. Then \begin{align*}\lambda(B) &\geq \lambda\big(M_r(D)\big) +2\\
& \geq \log_2\!\big(\!\dim M_r(D)\big)+3 \\
& \geq \log_2\!\big(n^2[D:k]/4\big)+3 \\
& = \log_2(\dim A)+1 \end{align*} by Corollary \ref{depthlowerboundcor} and the inductive hypothesis.

\vspace{2mm}\noindent Let $B$ be of type $(S2)$ or $(S3)$. Then $\dim A=2\dim B$ by Corollary \ref{classthm1cor}. So we have $$\lambda(B) \geq \log_2(\dim A/2) +1 = \log_2(\dim A)$$ by the inductive hypothesis. This proves the theorem.

\section{Proof of Corollary \ref{depthcor}}\label{proofdepthcor}

\noindent Let $A=M_n(D)$. Observe that $$\cd(A) \geq n^2/2+3n/2-2 -6\log_2 n \geq n^2/2-9$$ using Theorems \ref{length} and \ref{depth}. Rearranging, we have $n \leq \sqrt{2\cd(A)+18}$.

\section{Proof of Theorem \ref{depthfinitefields}}\label{proofdepthfinitefields}

% Let $k=\F_q$ and $A=M_n(\F_{q^m})$. We first show that $\lambda(A) \leq \Omega(m)+\lambda\big(M_n(k)\big)$. Let $\{a_i : i \in \Lambda\}$ be the set of prime divisors of $m$ (including multiplicities). The chain $$ A > M_n(\F_{q^{m/a_1}})>M_n(\F_{q^{m/(a_1a_2)}})>...>M_n(k)$$ is unrefinable of length $\Omega(m)$.

\noindent Let $k$ be a field such that, for every $r \in \N$, there exists a field extension of degree $r$ over $k$. Let $A=M_n(D)$. Let $D=D_0>D_1>...>k>0$ be an unrefinable chain of minimal length of division algebras. The chain $$ A > M_n(D_1)>...>M_n(k)$$ is unrefinable of length $\lambda(D)-1$. So $\lambda(A) \leq \lambda(D)-1 + \lambda\big(M_n(k)\big)$. It remains to show that $\lambda\big(M_n(k)\big) \leq \min\{2\Omega(n)+1,16\}$. Henceforth let $D=k$. That is, $A=M_n(k)$.

\vspace{2mm}\noindent Let $L$ be a field extension of $k$ of degree $n$. Consider the image of $L$ in $A$ under the left regular representation. Then $L$ is a strictly maximal subfield of $A$ by Corollary \ref{simplesubalgebrascor}. So $\lambda(A) \leq 2\Omega(n)+1$ by Lemma \ref{strictmax}. In particular, if $n$ is prime then $\lambda(A) \leq 3$ (we use this fact later). Note that $2\Omega(n)+1 \leq 16$ if $n \leq 6$.

\vspace{2mm}\noindent Now assume that $n$ is odd and $\geq 7$. Using Helfgott’s solution to the ternary Goldbach conjecture \cite{H}, we can write $n=p_1+p_2+p_3$ for some primes $p_1$, $p_2$ and $p_3$. The chain $$A > P_{p_1+p_2,p_3}(k) > M_{p_1+p_2}(k) \times M_{p_3}(k) > P_{p_1,p_2}(k) \times M_{p_3}(k)>\prod_{i=1}^3 M_{p_i}(k)$$ is unrefinable of length $4$ and so $\lambda(A) \leq 4+\sum_{i=1}^3\lambda\big(\hspace{-0.3mm}M_{p_i}(k)\big)=13$ since each $p_i$ is prime.

\vspace{2mm}\noindent If $n$ is even and $\geq 8$ then the chain $$ A > P_{n-1,1}(k)>M_{n-1}(k) \times k>M_{n-1}(k)$$ is unrefinable of length $3$ and so $\lambda(A) \leq 3+\lambda\big(\hspace{-0.3mm}M_{p_i}(k)\big)\leq 16$.

\section{Proof of Theorem \ref{dimthings}}\label{proofdimthings}

\noindent Let $A$ be an algebra. We first characterise when $l(A)=\dim A$. That is, we prove part $(i)$ of Theorem \ref{dimthings} in the following lemma.

\begin{lemma}\label{length=dim} $l(A)=\dim A$ if and only if $A/J(A)$ is a direct product of copies of $k$, $M_2(k)$ and quadratic field extensions of $k$.
\begin{proof} Let $A/J(A) = k^a \times M_2(k)^b \times \prod_{i=1}^m k_i $ for non-negative integers $a,b$ and $m$ and quadratic field extensions $k_i$ of $k$. Then $l(A)=\dim J(A)+a+bl\big(M_2(k)\big)+\sum_{i=1}^m l(k_i)$ by Corollary \ref{reduction}. Observe that the chain $$M_2(k)>T_2(k)>k^2>k>0$$ is unrefinable and so $l\big(M_2(k)\big)\geq 4$. Similarly, for each $i$, the chain $k_i>k>0$ is unrefinable and so $l(k_i)\geq 2 $. Since the dimension of an algebra is an upper bound for its length, we have $l\big(M_2(k)\big)= 4$ and $l(k_i)=2$ for each $i$. So $l(A)=\dim J(A)+a+4b+2m=\dim A$.

\vspace{2mm}\noindent Conversely, assume that $l(A)=\dim A$. Let $S \cong M_r(D)$ be a simple factor of $A/J(A)$, where $r$ is a positive integer and $D$ is a division algebra. Then $l(S)=\dim(S)$ by Corollary \ref{reduction}.

\vspace{2mm}\noindent If $r=1$ then $[D:k]=l(D) \leq \Omega([D:k])+1$ by Lemma \ref{multiplicativity} and so $[D:k]=1$ or $2$. If $[D:k]=2$ then $D$ must be a field since the $Z(D)$-dimension of $D$ is the square of an integer. Henceforth assume that $r>1$. Since $l(S)=\dim(S)$, there exists a maximal subalgebra $M$ of $S$ of codimension $1$. By Theorem \ref{classthm1}, $M$ is of type $(S1)$, $(S2)$ or $(S3)$. If $M$ is of type $(S2)$ or $(S3)$ then $\dim M$ divides $\dim S$, a contradiction. If $M$ is of type $(S1)$ then $M$ is conjugate to $P_{\alpha,n-\alpha}(D)$ for some positive integer $\alpha <n$. Then $\dim M=\dim S-\alpha(n-\alpha)\big[D:k]$ and so $n=2$, $\alpha=1$ and $[D:k]=1$.
\end{proof}
\end{lemma}

\noindent We next consider what happens when $\cd(A)=0$.

\begin{lemma}\label{simplestuff} Let $A$ be a simple algebra that is not a division algebra. Then $\cd(A)=0$ if and only if $A = M_2(k)$ and there does not exist a quadratic field extension of $k$.
\begin{proof} Let $A=M_n(D)$ for some integer $n>1$ and division algebra $D$. Let $\cd(A)=0$. Assume (for a contradiction) that $n \geq 3$. Any unrefinable chain of form $$A>...>M_3(k)>P_{2,	1}(k)>M_2(k) \times k >M_2(k)>T_2(k) >k^2>k >0$$ has length at most $l(A)-l\big(M_3(k)\big)+7=l(A)-1$ (using Theorem \ref{length}). This contradicts $\cd(A)=0$. So $n = 2$.

\vspace{2mm}\noindent Recall that $J\big(T_2(D)\big)=U_2(D)$, which has dimension $[D:k]$, and that $T_2(D)/U_2(D) \cong D^2$. So $l(A)-1 \geq l\big(T_2(D)\big)=2l(D)+[D:k]$ by Corollary \ref{reduction}. Let $D>D_1>...>k>0$ be an unrefinable chain of minimal length. Then the chain $$A > T_2(D) > D^2 > D > D_1>... ... > k >0$$ is unrefinable of length $\lambda(D)+3$. Hence $D = k$ since $\cd(A)=0$. That is, $A=M_2(k)$.

\vspace{2mm}\noindent Observe that $l(A)=4$ by Lemma \ref{length=dim}. Let $F$ be a quadratic field extension of $k$. Then $F$ embeds in $M_2(k)$ via the left regular representation. So the chain $$M_2(k) > F > k>0$$ is unrefinable of length $3$, which contradicts $\lambda(A)=l(A)$.

\vspace{2mm}\noindent Conversely, let $A=M_2(k)$ and assume that there does not exist a quadratic field extension of $k$. Once again, $l(A)=4$ by Lemma \ref{length=dim}. Let $M$ be a maximal subalgebra of $A$. Assume (for a contradiction) that $\lambda(M) < 3$. If $M$ is parabolic then $\lambda(M) \geq 3$ by Corollary \ref{depthlowerboundcor}, a contradiction. Hence, by Theorem \ref{classthm1}, $M=C_A(F)$ for some subfield $F$ of $A$. Since $\dim A=4$, we must have $[F:k]=2$, a contradiction.
\end{proof}
\end{lemma}

\noindent Recall that an algebra $A$ satisfies condition $(*)$ if $A/J(A) \cong M_2(k)^{\delta} \times \prod_{i=1}^m D_i$ for non-negative integers $\delta,m$ and division algebras $D_i$ such that
\begin{itemize}
\item $\cd(D_i)=0$ for each $1 \leq i \leq m$,
\item either $\delta=1$ and there does not exist a quadratic field extension of $k$ or $\delta=0$, and
\item any division subalgebra of $A/J(A)$ either embeds in $D_i$ for precisely one $i$ or is isomorphic to $k$.
\end{itemize}

\begin{lemma}\label{iioneway} Let $A$ be an algebra such that $\cd(A)=0$. Then $A$ satisfies condition $(*)$.
\begin{proof} Observe that $\cd\!\big(A/J(A)\big)=0$ since $\cd(A)=0$. Write $A/J(A) = \prod_{i=1}^m S_i$ by Wedderburn's Theorem, where each $S_i$ is a simple factor of $A/J(A)$. Observe that $$\lambda\big(A/J(A)\big) \leq \sum_{i=1}^m \lambda(S_i) \leq \sum_{i=1}^m l(S_i)=l\big(A/J(A)\big)$$ by Corollary \ref{reduction} and so $\cd(S_i)=0$ for each $i$. By Lemma \ref{simplestuff}, for each $i$, either $S_i$ is a division algebra or $S_i=M_2(k)$. So let us write $A/J(A) = M_2(k)^{\delta} \times \prod_{i=1}^m D_i $ for non-negative integers $m,\delta$ and division algebras $D_i$ of $k$.

\vspace{2mm}\noindent Assume (for a contradiction) that $\delta >1$. Any unrefinable chain of form $$A/J(A) > ... > M_2(k)^2 > M_2(k) > T_2(k)>k^2>k>0$$ has length at most $l\big(A/J(A)\big)-l\big(M_2(k)^2\big)+5 <l\big(A/J(A)\big)$ (using Lemma \ref{length=dim}). This contradicts $\cd\!\big(A/J(A)\big)=0$. So $\delta =0$ or $1$. If $\delta =1$ then there does not exist a quadratic field extension of $k$ by Lemma \ref{simplestuff}.

\vspace{2mm}\noindent Let $E$ be a division subalgebra of $A/J(A)$ with $[E:k]>1$. Since $E$ is simple it must embed in some simple factor of $A/J(A)$. % Do I need to justify this?
 If $[E:k]=2$ then $E$ is a quadratic field extension of $k$. So if $\delta=1$ then $[E:k]>2$ and $E$ cannot embed in $M_2(k)$. Assume (for a contradiction) that $E$ embeds in both $D_i$ and $D_j$ for $1 \leq i < j \leq m$. Any unrefinable chain of form $$D_i \times D_j > ... >E^2 > E > ... >k >0$$ has length at most $l(D_i \times D_j)-l(E)+1$ (using Corollary \ref{reduction}). Since $\cd(D_i \times D_j)=0$, we have $l(E)=1$ and so $E \cong k$, a contradiction. We have shown that $E$ embeds in $D_i$ for precisely one $i$.
\end{proof}
\end{lemma}

\noindent We now prove a partial converse to Lemma \ref{iioneway}.

\begin{lemma}\label{iioneway} Let $A$ be a semisimple algebra that satisfies condition $(*)$. Then $\cd(A)=0$.
\begin{proof} We write $A = M_2(k)^{\delta} \times \prod_{i=1}^m D_i$ where $\delta=0$ or $1$. We have $\lambda(A) \leq l(A)=4\delta+\sum_{i=1}^m l(D_i)$ by Corollary \ref{reduction} and Lemma \ref{length=dim}. We will show that $\lambda(A) \geq 4\delta+\sum_{i=1}^m l(D_i)$ by induction on $\lambda(A)$.

\vspace{2mm}\noindent Let $M$ be a maximal subalgebra of $A$. We first consider the case where $M$ is conjugate to $M_2(k)^{\delta} \times E_j \times \prod_{i \neq j} D_i$ for some $1\leq j \leq m$ where $E_j$ is a maximal subalgebra of $D_j$. Observe that $M$ is semisimple and satisfies condition $(*)$ and so $\lambda(M) \geq 4\delta+\sum_{i=1}^m l(D_i)-1$ by the inductive hypothesis.

\vspace{2mm}\noindent Since $A$ satisfies condition $(*)$, no two simple factors of $A$ are isomorphic unless they are both isomorphic to $k$. Hence, by Theorem \ref{classthm2}, the only other possibility for $M$ is when $\delta=1$ and $M$ is conjugate to $M_0 \times \prod_{i=1}^m D_i$ for some maximal subalgebra $M_0$ of $M_2(k)$. By Theorem \ref{classthm1}, since there are no quadratic field extensions of $k$, $M_0$ is conjugate to $T_2(k)$ in $M_2(k)$. Let $M_1:= k^2 \times \prod_{i=1}^m D_i$. Observe that $M_1$ satisfies condition $(*)$ and $M/J(M) \cong M_1$. Hence $\lambda(M) \geq \lambda(M_1)+1 \geq \sum_{i=1}^m l(D_i)+3$ by Lemma \ref{depthlowerbound} and the inductive hypothesis.

\vspace{2mm}\noindent We have shown that $l(M) \geq 4\delta+\sum_{i=1}^m l(D_i) -1$. So $\lambda(A) = l(A)$.
\end{proof}
\end{lemma}

\noindent A quiver is a directed graph that allows loops and multiple arrows between two vertices. Associated to a quiver $Q$ is a path algebra $kQ$ that is generated by the set of all paths in $Q$ with relations consistent with concatenation of paths. The ideal $kQ_+$ of $kQ$ is generated by all paths of length at least $1$. % See https://www.math.uni-bielefeld.de/~sek/kau/leit4.pdf
An admissible ideal of $kQ$ is a two-sided ideal $I$ of $kQ$ that satisfies $(kQ_+)^l \subseteq I \subseteq (kQ_+)^2$ for some integer $l \geq 2$.

\vspace{2mm}\noindent We complete the proof of part $(ii)$ of Theorem \ref{dimthings} in the following lemma. Recall that $A$ is \textit{basic split} if $A/J(A) \cong k^b$ for some $b \geq 0$ (this includes nilpotent algebras).

\begin{lemma}\label{basicsplit} Let $A$ be a finite dimensional basic split algebra. Then $l(A)=\lambda(A)=\dim A$.
\begin{proof} The case where $A$ is nilpotent is covered by Lemma \ref{nilpotent}. So assume that $A$ is not nilpotent. It follows from Wedderburn's Theorem that a non-nilpotent algebra $A$ is basic split if and only if all simple modules of $A$ are of dimension $1$. % Lemma 6 of https://uu.diva-portal.org/smash/get/diva2:770728/FULLTEXT03.pdf for alg. closed case
Hence $A$ is isomorphic to the quotient of the path algebra of a quiver by an admissible ideal (see $\S 4.1$ of \cite{B} or $\S 4.1$ of \cite{IS}). Then we can apply Proposition $4.2$ of \cite{IS} which tells us that every maximal subalgebra of $A$ has codimension $1$. We are done as all subalgebras of a basic split algebra are also basic split.
\end{proof}
\end{lemma}

\section{Proof of Theorem \ref{solvrad}}\label{proofsolvrad}

\noindent We first consider the case where $A=M_n(D)$ for some integer $n>1$ and division algebra $D$.

\vspace{2mm}\noindent Let $n_0:=n$ and define $n_i$ iteratively by $n_i:=\lfloor n_{i-1}/2 \rfloor$. Let $t$ be minimal such that $n_t=1$. In the proof of Theorem \ref{depth} we constructed an unrefinable chain \begin{equation}\label{importantchain} A > ... > M_n(k) > ... > M_{n_1}(k)>... > M_{n_2}(k) >... >M_{n_t}(k) > 0\end{equation} where the subchain between $A$ and $M_n(k)$ has length $\lambda(D)-1$ and, for each $0 \leq i < t$, the subchain between $M_{n_i}(k)$ and $M_{n_{i+1}}(k)$ has length $3$ (resp. $6$) if $n_i$ is even (resp. odd).

\vspace{2mm}\noindent Let $l_1$ denote the length of the chain $(\ref{importantchain})$. So $\lambda(A) \leq l_1$. Observe that $l_1-\lambda(D)$ depends only upon $n$. So define $f:\N \to \N$ by $f(n):=l_1-\lambda(D)$. It is not hard to see that $f(n) \leq 2n$ for all $n$. % Observe that $f(2)=3$, $f(3)=6=g(4)$, $f(5)=9=f(6)=f(8)$ and $f(7)=12=f(9)$. In the proof of Theorem \ref{depth} we showed that $f(n) \leq 6\log_2 n$ for all $n$. It follows that $2n\geq f(n)$.

\vspace{2mm}\noindent If $n\geq 3$ then \begin{align*}\dim A &=n^2[D:k] \\ & \leq 9\big(n(n-1)[D:k]/2- 2\big) \\ & \leq 9\big(n(n-1)[D:k]/2-2+2n-f(n)\big) \\ &\leq 9\big(l(A)-\lambda(D) - f(n)\big) \\ & \leq 9\cd(A)\end{align*} using Theorem \ref{length}.

\vspace{2mm}\noindent If $n=2$ and $D \not\cong k$ then $$\dim A =4[D:k] \leq 9\big(l(A)-\lambda(D)-3\big) \leq 9\cd(A)$$ using Theorem \ref{length} and since $l(D)\geq 2$. Note that $f(2)=3$.

\vspace{2mm}\noindent We have thus far shown that $\dim A \le 9 \cd(A)$ for all simple algebras $A=M_n(D)$ with $n >1$ except for $M_2(k)$.

\vspace{2mm}\noindent We now consider the general case. That is, let $A$ be any algebra. Write $A/R(A) = \prod_{i=1}^m A_i$ where each $A_i$ is a simple subalgebra of $A$ (and is not a division algebra).

\vspace{2mm}\noindent Since the length function is additive, and the depth function is sub-additive, we see that
the chain difference is super-additive, namely, for an ideal $I \subseteq A$ we have
\[
\cd(A) \ge \cd(I) + \cd(A/I).
\]
Then $\cd(A) \ge \cd\!\big(\hspace{-0.2mm}A/R(A)\big) \ge \sum_{i=1}^m \cd(A_i)$
by super-additivity. If $A_i \not\cong M_2(k)$ for each $i$ then we are done. So, without loss of generality, let $j \in \{1,...,m\}$ be such that $A_i \cong M_2(k)$ for all $1 \leq i \leq j$ and $A_i \not\cong M_2(k)$ for all $j < i \leq m$.

\vspace{2mm}\noindent Assume that $j > 1$. The chain $$\prod_{i=1}^j A_i  > \prod_{i=1}^{j-1} A_i >...>A_1$$ is unrefinable of length $j-1$ and so $\lambda\big(\prod_{i=1}^j A_i\big) \leq j-1+ \dim M_2(k)=j+3$. Note that $l\big(\prod_{i=1}^j A_i\big) = 4j$ by Theorem \ref{length}. So $\dim \prod_{i=1}^j A_i =4j \leq 9\cd\!\big(\prod_{i=1}^j A_i\big)$ since $j > 1$. Then $\dim A/R(A) \le 9 \cd(A)$ by super-additivity.

\vspace{2mm}\noindent We next assume that $j=1$ and $m >1$. Write $A_2 = M_n(D)$ and consider the unrefinable chain $(\ref{importantchain})$ of $A_2$. We already showed that $\dim A_2 \leq 9\big(l(A_2)-\lambda(D)-f(n)\big)$. By construction, the chain $(\ref{importantchain})$ of $A_2$ passes through $M_2(k)$. Observe that the chain $$M_2(k)^2 >M_2(k)>T_2(k)>k^2>k>0$$ is unrefinable of length $5$ and hence $\lambda(A_1 \times A_2) \leq \lambda(D)+f(n)+1$. Note that $l(A_1\times A_2)=4+l(A_2)$ by Theorem \ref{length}. It follows that $$\dim (A_1 \times A_2) =4+\dim A_2 \leq 4+9\big(l(A_2)-\lambda(D)-f(n)\big) %\leq 4+9\big(\!\cd(A_1 \times A_2) -3\big)
\leq 9\cd(A_1 \times A_2).$$ Then $\dim A/R(A) \le 9 \cd(A)$ by super-additivity.

\vspace{2mm}\noindent At this stage we have shown that $\dim A/R(A) \le 9 \cd(A)$ unless $A/R(A)\cong M_2(k)$.

\vspace{2mm}\noindent Finally, we assume that $A/R(A)\cong M_2(k)$ and there exists a quadratic field extension of $k$. Then $A$ does not satisfy condition $(*)$ and so $\cd(A) \geq 1$ by Theorem \ref{dimthings}$(ii)$. So $\dim A/R(A) =4 \le 9 \cd(A)$. This completes the proof.

\section{Proof of Theorem \ref{ld}}

\noindent First note that, for various fields $k$ (including all number fields and all finite fields), Theorem \ref{depthfinitefields}
immediately shows that, for $k$-algebras $A$, $\dim A$ is not bounded above in terms of $\lambda(A)$.
We now prove the stronger result, that $\dim A$ is not bounded above in terms of $l(A)$, as stated in part $(i)$ of Theorem \ref{ld}.

\vspace{2mm}\noindent Let $k = \Q$. For each prime $p$ there exists a Galois extension $F(p)$ of $k$
with $\Gal(F(p)/k) = C_p$, the cyclic group of order $p$. Set $A = F(p)$.

\vspace{2mm}\noindent By Lemma \ref{multiplicativity} we have $l(A) = l(C_p)+1 = 2$ for all $p$.
On the other hand $\dim A = |\Gal(F(p)/k)| = |C_p| = p$, which tends to infinity as $p \to \infty$.
Part $(i)$ of Theorem \ref{ld} follows.

\vspace{2mm}\noindent Let us now prove part $(ii)$ of the theorem. Let $k$ be a field such that $[\overline{k}:k]<\infty$.

\vspace{2mm}\noindent Let $A$ be an algebra. Write $A/J(A) = \prod_{i=1}^m M_{n_i}(D_i)$. Without loss of generality we may assume that $\dim M_{n_1}(D_1) \geq \dim M_{n_i}(D_i)$ for all $1 \leq i \leq m$. By Corollary \ref{depthlowerboundcor} we have
\[
\lambda(A) \geq \lambda\big(M_{n_i}(D_i)\big)+m-1
\]
for all $1 \leq i \leq m$.

\vspace{2mm}\noindent Since $[\overline{k}:k]<\infty$, by Corollary VIII.$9.2$ of \cite{L}, either $k=\overline{k}$ or $[\overline{k}:k]=2$ and $k$ has characteristic $0$.

\vspace{2mm}\noindent Suppose first that $k=\overline{k}$, namely $k$ is algebraically closed. Then $D_i=k$ for each $i$. By Theorem \ref{depth} it follows that
\[
\lambda(A) \ge 3 \log_2 n_1 + m.
\]
This yields
\[
2^{\lambda(A)} \geq n_1^3 \cdot 2^m.
\]
Hence $n_1 \le 2^{\lambda(A)/3 - m/3}$ and
\[
\dim A/J(A) = \sum_{i=1}^m n_i^2 \le m \cdot n_1^2 \le m 2^{-2m/3} \cdot 2^{2\lambda(A)/3}.
\]
Note that $m 2^{-2m/3} \le 1$ for all $m$, hence
\[
\dim A/J(A) \le 2^{2\lambda(A)/3}.
\]
Next suppose that $[\overline{k}:k]=2$. By Theorem \ref{depth} we have
\[
\lambda(A) \ge \log_2\!\big(\!\dim M_{n_1}(D_1)\big) + m.
\]

\vspace{2mm}\noindent Rearranging, this yields
\[
\dim M_{n_1}(D_1) \le 2^{\lambda(A) - m}.
\]
Hence
\[
\dim A/J(A) = \sum_{i=1}^m \dim M_{n_i}(D_i) \le m \dim M_{n_1}(D_1) \le m2^{-m} \cdot 2^{\lambda(A)}
\]
Note that $m 2^{-m}\le1$ for all $m$. Hence
\[
\dim A/J(A) \le 2^{\lambda(A)}.
\]
This completes the proof when $J(A)=0$.

\vspace{2mm}\noindent Let $g:\N \to \N$ be defined by $g(x):=2^x$. We have thus far shown that $\dim A/J(A) \le g(\lambda(A))$ (for any field $k$ such that $[\overline{k}:k]<\infty$). Let $f(x):=2g(x)^2$ for all $x \in \N$. We will show that $\dim A \le f(\lambda(A))$ by induction on $\lambda(A)$.

\vspace{2mm}\noindent If $\lambda(A)=1$ then $A\cong k$ and certainly $\dim A \leq f(\lambda(A))$. Let $B$ be a maximal subalgebra of $A$ such that $\lambda(B)=\lambda(A)-1$. If $B$ is of semisimple type then $J(A) \subseteq B$ and so $$\dim A \leq \dim B +\dim A/J(A) \leq f(\lambda(B))+g(\lambda(A)) \leq f(\lambda(A))$$ by the inductive hypothesis. It remains to consider the case where $B$ is of split type. We will need the following lemma.

\begin{lemma}\label{modulelemma} Let $S$ be a semisimple algebra and let $M$ be a simple $S$-bimodule. Then $\dim M \leq (\dim S)^2$.
\begin{proof} Let $S^{op}$ denote the opposite algebra of $S$. By Proposition $10.1$ of \cite{P}, there is an equivalence between the category of $S$-bimodules and the category of left $S \otimes_k S^{op}$-modules. Precisely, $M$ has the structure of a simple left $S \otimes_k S^{op}$-module that satisfies $(s \otimes s')m=(sm)s'=s(ms')$ for $s,s' \in S$ and $m \in M$.

\vspace{2mm}\noindent Observe that $S \otimes_k S^{op}$ is a semisimple algebra and so $S \otimes_k S^{op} \cong \prod_{i=1}^t M_{r_i}(\Delta_i)$ for some division algebras $\Delta_i$ by Wedderburn's Theorem. It is well known (see Proposition $2.3$ of \cite{P} or Theorem VIII.$2$ of \cite{A}) that any simple left module of $\prod_{i=1}^tM_{r_i}(\Delta_i)$ is isomorphic to $(\Delta_i)^{r_i}$ for some $i=1,...,t$. % https://math.stackexchange.com/questions/1641937/show-that-all-simple-modules-of-m-nd-are-isomorphic-to-dn-where-d?noredirect=1&lq=1    or    https://math.stackexchange.com/questions/874233/simple-m-nd-module-with-d-a-division-ring    or    https://math.stackexchange.com/questions/782690/simple-module-over-matrix-rings        or       Proposition 8 of https://uu.diva-portal.org/smash/get/diva2:770728/FULLTEXT03.pdf
Hence $\dim M \leq \dim \!\big(\hspace{-0.1mm}S \otimes_k S^{op}\big) = (\dim S)^2$.
\end{proof}
\end{lemma}
% If $S=\Prod_{i=1}^m M_{n_i}(D_i)$ then observe that $S \otimes_k S^{op}=\prod_{i=1}^m\prod_{j=1}^m M_{n_in_j}(D_i \otimes D_j^{op})$. So each simple component of $S \otimes_k S^{op}$ is contained in $M_{n_in_j}(D_i \otimes D_j^{op})$ for some $i, j$.

\noindent Observe that $k$ is perfect since it is either algebraically closed or it has characteristic $0$. So $A/J(A)$ is a separable algebra. Then there is a semisimple subalgebra $A_0$ of $A$ such that $A=A_0 \oplus J(A)$ (Wedderburn's Principal Theorem). By Theorem $0.1$ of \cite{IS}, $B$ is conjugate to $A_0 \oplus H$ for some maximal $A_0$-subbimodule $H$ of $J(A)$. So $J(A)/H$ is a non-trivial simple $A_0$-bimodule. Then we have $$\dim A=\dim B+\dim J(A)/H \leq \dim B+(\dim A_0)^2 \leq f(\lambda(B))+(g(\lambda(A)))^2 \leq f(\lambda(A))$$ by Lemma \ref{modulelemma} and the inductive hypothesis.

\vspace{2mm}\noindent This completes the proof of part $(ii)$, with $f:\N \to \N$ as a suitable function.

\end{document}